\documentclass{amsart}
\usepackage{amssymb}
\usepackage{hyperref}

\newtheorem{thm}{Theorem}[section]

\newtheorem{lem}[thm]{Lemma}
\newtheorem{prop}[thm]{Proposition}
\theoremstyle{definition}
\newtheorem{defn}[thm]{Definition}
\theoremstyle{remark}
\newtheorem{rem}[thm]{Remark}

\numberwithin{equation}{section}
\numberwithin{thm}{section}

\newcommand{\Z}{{\mathbb{Z}}}
\newcommand{\C}{{\mathbb{C}}}
\newcommand{\R}{{\mathbb{R}}}

\newcommand{\T}{{\mathbb{T}}}

\newcommand{\eps}{{\varepsilon}}

\newcommand{\lsm}{\lesssim}

\DeclareMathOperator*{\wlim}{w-lim}

\newcommand{\qtq}[1]{\quad\text{#1}\quad}

\newcommand{\cha}{\chi_n^0}

\newcommand{\chc}{\chi_n^2}
\newcommand{\chd}{\chi_n^3}
\newcommand{\che}{\chi_n^4}

\newcommand{\pmn}{P_{\le N_n}}

\newcommand{\pml}{ P_{\le N_n}^{L_n}}
\newcommand{\tun}{\tilde u_n}

\newcommand{\Po}{\mathcal P}
\newcommand{\pmll}{P_{\le 2N_n}^{L_n}}
\newcommand{\btn}{\mathbb{T}_n}

\let\Im=\undefined\DeclareMathOperator*{\Im}{Im}

\newcounter{smalllist}

\title[Symplectic non-squeezing for the cubic NLS on the line]{Symplectic non-squeezing\\for the cubic NLS on the line}

\author[R. Killip]{Rowan Killip}
\address{Department of Mathematics, University of California, Los Angeles, CA 90095}%
\email{killip@math.ucla.edu}

\author[M. Visan]{Monica Visan}
\address{Department of Mathematics, University of California, Los Angeles, CA 90095}
\email{visan@math.ucla.edu}

\author[X. Zhang]{Xiaoyi Zhang}
\address{Department of Mathematics, University of Iowa, Iowa City, IA 52242}%
\email{xiaoyi-zhang@uiowa.edu}
{\normalsize }
\begin{document}

\begin{abstract}
We prove symplectic non-squeezing for the cubic nonlinear Schr\"odinger equation on the line via finite-dimensional approximation.
\end{abstract}

\maketitle

\section{Introduction}

The main result of this paper is a symplectic non-squeezing result for the cubic nonlinear Schr\"odinger equation on the line:
\begin{align*}\label{nls}\tag{NLS}
iu_t+\Delta u= \pm |u|^2 u.
\end{align*}
We consider this equation for initial data in the underlying symplectic Hilbert space $L^2(\R)$.  For this class of initial data, the equation is globally well-posed in both the defocusing and focusing cases, that is, with $+$ and $-$ signs in front of the nonlinearity, respectively.  Correspondingly, we will be treating the defocusing and focusing cases on equal footing.

Our main result is the following:

\begin{thm}[Non-squeezing for the cubic NLS]\label{thm:nsqz}
Fix $z_*\in L^2(\R)$, $l\in L^2(\R)$ with $\|l\|_2=1$, $\alpha\in \C$, $0<r<R<\infty$, and $T>0$.  Then there exists $u_0\in B(z_*, R)$ such that the solution $u$ to \eqref{nls} with initial data $u(0)=u_0$ satisfies
\begin{align}\label{th:1}
|\langle l, u(T)\rangle-\alpha|>r.
\end{align}
\end{thm}

Colloquially, this says that the flow associated to \eqref{nls} does not carry any ball of radius $R$ into any cylinder whose cross-section has radius $r<R$.  Note that it is immaterial where the ball and cylinder are centered; however, it is essential that the cross-section of the cylinder is defined with respect to a pair of canonically conjugate coordinates.

The formulation of this result is dictated by the non-squeezing theorem of Gromov, \cite[\S0.3A]{Gromov}, which shows the parallel assertion for \emph{any} symplectomorphism of finite-dimensional Hilbert spaces.  At the present time, it is unknown whether this general assertion extends to the infinite-dimensional setting. 

Non-squeezing has been proved for a number of PDE models; see \cite{Bourg:approx,Bourg:aspects,CKSTT:squeeze,HongKwon,KVZ:nsqz2d,Kuksin,Mendelson,Roumegoux}.  We have given an extensive review of this prior work in our paper \cite{KVZ:nsqz2d} and so will not repeat ourselves here.  Rather, we wish to focus on our particular motivations for treating the model \eqref{nls}.

With the exception of \cite{KVZ:nsqz2d}, which considers the cubic NLS on $\R^2$, all the papers listed above considered the non-squeezing problem for equations posed on tori.  One of the initial goals for the paper \cite{KVZ:nsqz2d} was to treat (for the first time) a problem in infinite volume.  Moreover, we sought also to obtain a first unconditional result where the regularity required to define the symplectic form coincided with the scaling-critical regularity for the equation.  

Many of the central difficulties encountered in \cite{KVZ:nsqz2d} stem from the criticality of the problem considered there, to the point that they obscure the novel aspects associated to working in infinite volume.  One of our main motivations in writing this paper is to elaborate our previous approach in a setting unburdened by the specter of criticality.  In this way, we hope also to provide a more transparent framework for attacking (subcritical) non-squeezing problems in infinite volume.

In keeping with the expository goal just mentioned, we have elected here to treat a single model, namely, the cubic NLS in one dimension.  What follows applies equally well to any mass-subcritical NLS in any space dimension --- it is simply a matter of adjusting the H\"older/Strichartz exponents appropriately.

Let us now briefly outline the method of proof.  Like previous authors, the goal is to combine a suitable notion of finite-dimensional approximation with Gromov's theorem in that setting.  The particular manner in which we do this mirrors \cite{KVZ:nsqz2d}, but less so other prior work.
 
In the presence of a frequency truncation, NLS on a torus (of, say, large circumference) becomes a finite-dimensional system and so is non-squeezing in the sense of Gromov.  In particular, there is an initial datum $u_0$ in the ball of radius $R$ about $z_*$ so that the corresponding solution $u$ obeys \eqref{th:1} at time $T$.  We say that $u$ is a \emph{witness} to non-squeezing.

Now choosing a sequence of frequency cutoff parameters $N_n\to\infty$ and a sequence of circumferences $L_n\to\infty$, Gromov guarantees that there is a sequence of witnesses $u_n$.  Our overarching goal is to take a ``limit'' of these solutions and so obtain a witness to non-squeezing for the full (untruncated) model on the whole line.   This goal is realized in two steps: (i) Removal of the frequency cutoff for the problem in infinite volume; see Section~\ref{S:4}. (ii) Approximation of the frequency-truncated model in infinite volume by that on a large torus; see Section~\ref{S:5}.  The frequency truncation is essential for the second step since it enforces a form of finite speed of propagation.

The principal simplifications afforded by working in the subcritical case appear in the treatment of step (i); they are two-fold.  First, the proof of large-data space-time bounds for solutions to \eqref{nls} is elementary and applies also (after trivial modifications) to the frequency-truncated PDE.  This is not true for the critical problem.  Space-time bounds for the mass-critical NLS is a highly nontrivial result of Dodson \cite{Dodson:3,Dodson:1,Dodson}; moreover, the argument does not apply in the frequency-truncated setting because the truncation ruins the monotonicity formulae at the heart of his argument.  For a proof of uniform space-time bounds for suitably frequency-truncated cubic NLS on $\R^2$, see Section~4 in \cite{KVZ:nsqz2d}. 

The second major simplification relative to \cite{KVZ:nsqz2d} appears when we prove wellposedness in the weak topology on $L^2$.  Indeed, the reader will notice that the statement of Theorem~\ref{T:weak wp} here is essentially identical to that of Theorem~6.1 in \cite{KVZ:nsqz2d}; however, the two proofs bear almost no relation to one another.  Here we exploit the fact that bounded-mass solutions are compact on bounded sets in space-time in a scaling-critical $L^p$ norm; this is simply not true in the mass-critical case.  See Section~4 for further remarks on this topic.

\subsection*{Acknowledgements} This material is based on work supported by the National Science Foundation under Grant No. 0932078000 while the authors were in residence at the MSRI in Berkeley, California, during the Fall 2015 semester.  R.~K. and M.~V. are grateful to IHES, France, which provided the perfect environment for the completion of this work. 

This work was partially supported by a grant from the Simons Foundation (\#342360 to Rowan Killip).
R.~K. was further supported by NSF grants DMS-1265868 and DMS-1600942.  M.~V. was supported by NSF grant DMS-1500707. X.~Z. was supported by a Simons Collaboration grant.

\section{Preliminaries}

Throughout this paper, we will write the nonlinearity as $F(u):=\pm|u|^2u$.

\begin{defn}[Strichartz spaces] We define the Strichartz norm of a space-time function via
$$
\| u \|_{S(I\times\R)} := \| u\|_{C^{ }_t L^2_x (I\times\R)} + \| u\|_{L_t^4 L^\infty_x (I\times\R)}
$$
and the dual norm via
$$
	\| G \|_{N(I\times\R)} := \inf_{G=G_1+G_2} \| G_1\|_{L^1_t L^2_x (I\times\R)} + \| G_2 \|_{L^{4/3}_t L^{1 }_x (I\times\R)}.
$$
We define Strichartz spaces on the torus analogously.
\end{defn}

The preceding definition permits us to write the full family of Strichartz estimates in a very compact form; see \eqref{StrichartzIneq} below.  The other basic linear estimate that we need is local smoothing; see \eqref{LocalSmoothing} below.

\begin{lem}[Basic linear estimates]
Suppose $u:\R\times\R\to\C$ obeys
$$
(i\partial_t +\Delta) u = G.
$$
Then for every $T>0$ and every $R>0$,
\begin{align}\label{StrichartzIneq}
\| u\|_{S([0,T]\times\R)} \lsm \| u(0) \|_{L^2_x} + \| G \|_{N([0,T]\times\R)},
\end{align}
\begin{align}\label{LocalSmoothing}
\| u\|_{L^2_{t,x}([0,T]\times[-R,R])}
	\lsm R^{1/2} \Bigl\{ \| |\nabla|^{-1/2} u(0) \|_{L^2_x} + \| |\nabla|^{-1/2} G \|_{L^1_t L^2_x([0,T]\times\R)} \Bigr\}. 
\end{align}
\end{lem}

Let $m_{\le 1}:\R\to[0,1]$ be smooth, even, and obey
$$
m_{\le 1}(\xi) = 1 \text{ for $|\xi|\leq 1$} \qtq{and} m_{\le 1}(\xi) = 0 \text{ for $|\xi|\geq 2$.}
$$
We define Littlewood--Paley projections onto low frequencies according to
\begin{align}\label{E:LP defn}
\widehat{ P_{\leq N} f }(\xi) := m_{\le 1}(\xi/N) \hat f(\xi)
\end{align}
and then projections onto individual frequency bands via
\begin{align}\label{E:LP defn'}
f_N := P_N f := [ P_{\leq N} - P_{\leq N/2} ] f .
\end{align}

\section{Well-posedness theory for several NLS equations}

In the course of the proof of Theorem~\ref{thm:nsqz}, we will need to consider the cubic NLS both with and without frequency truncation.  To consider both cases simultaneously, we consider the following general form:
\begin{align}\label{eq:1}
iu_t+\Delta u=  \Po F(\Po u), 
\end{align}
where $\Po$ is either the identity or the projection to low frequencies $P_{\leq N}$ for some $N\in 2^\Z$.  For the results of this section, it only matters that $\Po$ is $L^p$ bounded for all $1\leq p\leq\infty$.

\begin{defn}(Solution)\label{D:solution}
Given an interval $I\subseteq\R$ with $0\in I$, we say that $u:I\times\R\to\C$ is a \emph{solution} to \eqref{eq:1} with initial data $u_0\in L^2(\R)$ at time $t=0$ if $u$ lies in the classes $C^{ }_t L^2_x(K\times\R)$ and  $L^6_{t,x}(K\times\R)$ for any compact $K\subset I$ and obeys the Duhamel formula
\begin{align*}
u(t) = e^{it\Delta} u_0 - i \int_{0}^{t} e^{i(t-s)\Delta} \Po F\bigl(\Po u(s) \bigr) \, ds
\end{align*}
for all $t \in I$.
\end{defn}

Such solutions to \eqref{eq:1} are unique and conserve both mass and energy:
\begin{align*}
\int_{\R} |u(t,x)|^2 \,dx \qtq{and}
E(u(t)):=\int_{\R}\tfrac 12 |\nabla u(t,x)|^2 \pm\tfrac14 |\Po u(t,x)|^4 \,dx,
\end{align*}
respectively.  Indeed, \eqref{eq:1} is the Hamiltonian evolution associated to $E(u)$ through the standard symplectic structure:
$$
\omega:L^2(\R)\times L^2(\R)\to \R \qtq{with} \omega(u,v)=\Im \int_\R u(x)\bar v(x)\, dx.
$$

The well-posedness theory for \eqref{eq:1} reflects the subcritical nature of this equation with respect to the mass.  We record this classical result without a proof.

\begin{lem}[Well-posedness of \eqref{eq:1}]\label{lm:loc}
Let $u_0\in L^2(\R)$ with $\|u_0\|_2\le M$.  There exists a unique global solution $u:\R\times\R\to \C$ to \eqref{eq:1} with initial data $u(0)=u_0$.  Moreover, for any $T>0$,
\begin{align*}
 \|u\|_{S([0,T]\times\R)}\lsm_T M.
\end{align*}
If additionally $u_0\in H^1(\R)$, then
\begin{align*}
 \|\partial_x u\|_{S([0,T]\times\R)} \lsm_{M,T} \| \partial_x u_0 \|_{L^2(\R)}.
\end{align*} 
\end{lem}

\section{Local compactness and well-posedness in the weak topology}\label{S:4}

The arguments presented in this section show that families of solutions to \eqref{nls} that are uniformly bounded in mass are precompact in $L^p_{t,x}$ for $p<6$ on bounded sets in space-time.  Furthermore, we have well-posedness in the weak topology on $L^2$; specifically, if we take a sequence of solutions $u_n$ to \eqref{nls} for which the initial data $u_n(0)$ converges \emph{weakly} in $L^2(\R)$, then $u_n(t)$ converges weakly at all times $t\in\R$.  Moreover, the pointwise in time weak limit is in fact a solution to \eqref{nls}.

Justification of the assertions made in the first paragraph can be found within the proof of \cite[Theorem~1.1]{Nakanishi:SIAM}; however, this is not sufficient for our proof of symplectic non-squeezing.  Rather, we have to prove a slightly stronger assertion that allows the functions $u_n$ to obey different equations for different $n$; specifically,
\begin{align}\label{eqpn}
i\partial_t u_n+\Delta u_n=P_{\le N_n} F(P_{\le N_n}u_n) 
\end{align}
where $N_n\to \infty$; see Theorem~\ref{T:weak wp} below.  For the sake of completeness, we give an unabridged proof of this theorem, despite substantial overlap with the arguments presented in~\cite{Nakanishi:SIAM}.

What follows adapts easily to the setting of any mass-subcritical NLS.  It does \emph{not} apply at criticality: compactness fails in any scale-invariant space-time norm (even on compact sets).  Nevertheless, well-posedness in the weak topology on $L^2$ does hold for the mass-critical NLS; see \cite{KVZ:nsqz2d}.  Well-posedness in the weak topology has also been demonstrated for some energy-critical models; see \cite{BG99, KMV:cq}.  In all three critical results \cite{BG99, KMV:cq, KVZ:nsqz2d}, the key to overcoming the lack of compactness is to employ concentration compactness principles.  We warn the reader, however, that the augments presented in \cite{KVZ:nsqz2d} are far more complicated than would be needed to merely verify well-posedness in the weak topology for the mass-critical NLS.  In that paper, we show non-squeezing and so (as here) we were compelled to consider frequency-truncated models analogous to \eqref{eqpn}.  Due to criticality, this change has a profound effect on the analysis; see \cite{KVZ:nsqz2d} for further discussion.

Simple necessary and sufficient conditions for a set $\mathcal F\subseteq L^p(\R^n)$ to be precompact were given in \cite{Riesz}, perfecting earlier work of Kolmogorov and Tamarkin.  In addition to boundedness (in $L^p$ norm), the conditions are tightness and equicontinuity:
$$
\lim_{R\to\infty} \sup_{f\in\mathcal F} \|f\|_{L^p(\{|x|>R\})} =0
	\qtq{and} \lim_{h\to 0} \sup_{f\in\mathcal F} \|f(\cdot+h)-f(\cdot)\|_{L^p(\R^n)}  =0,
$$
respectively. The basic workhorse for equicontinuity in our setting is the following lemma:

\begin{lem}\label{L:workhorse}
Fix $T>0$ and suppose $u:[-2T,2T]\times\R\to\C$ obeys
\begin{align}
\| u \|_{\tilde S} := \| u \|_{L^\infty_t L^2_x([-2T,2T]\times \R)} + \| (i\partial_t+\Delta) u \|_{L^2_{t,x}([-2T,2T]\times\R)} < \infty.
\end{align}
Then
\begin{align}\label{22Holder}
 \| u(t+\tau,x+y) - u(t,x) \|_{L^2_{t,x}([-T,T]\times[-R,R])} \lesssim_{R,T} \bigl\{ |\tau|^{1/5} + |y|^{1/3} \bigr\} \| u \|_{\tilde S},
\end{align}
uniformly for $|\tau|\leq T$ and $y\in \R$.
\end{lem}

\begin{proof}
It is sufficient to prove the result when $y=0$ and when $\tau=0$; the full result then follows by the triangle inequality.  In both cases, we use \eqref{LocalSmoothing} to estimate the high-frequency portion as follows:
\begin{align}\label{22uhi}
\| u_{>N}(t+\tau,x+y) - u_{>N}(t,x) \|_{L^2_{t,x}([-T,T]\times[-R,R])} \lesssim R^{\frac12} N^{-\frac12}(1+T^{\frac12}) \| u \|_{\tilde S}.
\end{align}

Next we turn to the low-frequency contribution. Consider first the case $\tau=0$.  By Bernstein's inequality,
$$
\| u_{\leq N}(t,x+y) - u_{\leq N}(t,x) \|_{L^2_x(\R)} \lesssim N|y| \, \| u(t) \|_{L^2_x(\R)} \lesssim N|y| \, \| u \|_{\tilde S}.
$$
Therefore, setting $N=|y|^{-2/3}$, integrating in time, and using \eqref{22uhi}, we obtain
\begin{align}\label{22ulo2}
 \| u(t,x+y) - u(t,x) \|_{L^2_{t,x}([-T,T]\times[-R,R])} \lesssim (R^{\frac12}+(RT)^{\frac12}+T^{\frac12} ) |y|^{\frac13}  \| u \|_{\tilde S}.
\end{align}

Consider now the case $y=0$.   Here it is convenient to use the Duhamel representation of $u$:
$$
u(t) = e^{it\Delta} u(0) - i \int_0^t e^{i(t-s)\Delta} (i\partial_s+\Delta) u(s)\,ds.
$$
To exploit this identity, we first observe that 
$$
\bigl\| P_{\leq N} \bigl[ e^{i\tau\Delta} - 1] e^{it\Delta} u(0) \bigr\|_{L^2_{x}(\R)} \lesssim N^2 |\tau| \| u(0) \|_{L^2_{x}(\R)}.
$$
Then, by the Duhamel representation and the Strichartz inequality, we obtain
\begin{align*}
\bigl\| u_{\leq N}(t+\tau) - u_{\leq N}(t) \bigr\|_{L^2_x(\R)}
	&\lesssim N^2 |\tau| \| u(0) \|_{L^2_{x}(\R)} + \| (i\partial_t+\Delta)u\|_{L^1_t L^2_x([t,t+\tau]\times\R)} \\
&\lesssim \{ N^2 |\tau| + |\tau|^{\frac12} \}  \| u \|_{\tilde S}.
\end{align*}
Combining this with \eqref{22uhi} and choosing $N=|\tau|^{-2/5}$ yields
\begin{equation}\label{22ulo1}
\bigl\| u(t+\tau) - u(t) \bigr\|_{L^2_{t,x}([-T,T]\times[-R,R])}
	\lesssim (R^{\frac12}+(RT)^{\frac12}+ T^{\frac12}) (|\tau|^{\frac15} +|\tau|^{\frac12}) \| u \|_{\tilde S}.
\end{equation}

This completes the proof of the lemma.
\end{proof}

\begin{prop}\label{P:compact}
Let $u_n:\R\times\R\to\C$ be a sequence of solutions to \eqref{eqpn} corresponding to some sequence of $N_n>0$.  We assume that
\begin{equation}\label{apmassbound}
M:= \sup_n \| u_n(0) \|_{L^2_x(\R)} < \infty.
\end{equation}
Then there exist $v:\R\times\R\to\C$ and a subsequence in $n$ so that
\begin{align}\label{E:P:compact}
\lim_{n\to\infty} \| u_n - v \|_{L^p_{t,x}([-T,T]\times[-R,R])} = 0,
\end{align}
for all $R>0$, all $T>0$, and all $1\leq p <6$.
\end{prop}

\begin{proof}
A simple diagonal argument shows that it suffices to consider a single fixed pair $R>0$ and $T>0$.  In what follows, implicit constants will be permitted to depend on $R$, $T$, and $M$.

In view of Lemma~\ref{lm:loc} and \eqref{apmassbound}, we have
\begin{equation}\label{apbound}
\sup_n \| u_n \|_{S([-4T,4T]\times\R)} \lesssim 1.
\end{equation}
Consequently, if we define $\chi:\R^2\to[0,1]$ as a smooth cutoff satisfying
$$
\chi(t,x)= \begin{cases} 1 &:\text{ if } |t|\leq T \text{ and } |x|\leq R, \\
0 &:\text{ if } |t| > 2T \text{ or } |x|> 2R, \end{cases}
$$
then the sequence $\chi u_n$ is uniformly bounded in $L^2_{t,x}(\R\times\R)$.  Moreover, by Lemma~\ref{L:workhorse} and \eqref{apbound}, it is also equicontinuous.  As it is compactly supported, it is also tight.  Thus, $\{\chi u_n\}$ is precompact in $L^2_{t,x}(\R\times\R)$ and so there is a subsequence such that \eqref{E:P:compact} holds with $p=2$.  That it holds for the other values $1\leq p <6$ then follows from H\"older and \eqref{apbound}, which implies that $\{\chi u_n\}$ is uniformly bounded in $L^6_{t,x}([-T,T]\times\R)$.
\end{proof}

\begin{thm}\label{T:weak wp}
Let $u_n:\R\times\R\to\C$ be a sequence of solutions to \eqref{eqpn} corresponding to a sequence $N_n\to\infty$.  We assume that
\begin{equation}
u_n(0) \rightharpoonup u_{\infty,0}  \quad\text{weakly in $L^2(\R)$}
\end{equation}
and define $u_\infty$ to be the solution to \eqref{nls} with $u_\infty(0)=u_{\infty,0}$.  Then
\begin{align}\label{E:P:weak wp}
u_n(t) \rightharpoonup u_\infty(t)  \quad\text{weakly in $L^2(\R)$}
\end{align}
for all $t\in\R$.
\end{thm}

\begin{proof}
It suffices to verify \eqref{E:P:weak wp} along a subsequence; moreover, we may restrict our attention to a fixed (but arbitrary) time window $[-T,T]$.  Except where indicated otherwise, all space-time norms in this proof will be taken over the slab $[-T,T]\times\R$.

Weak convergence of the initial data guarantees that this sequence is bounded in $L^2(\R)$ and so by Lemma~\ref{lm:loc} we have
\begin{equation}\label{apb}
\sup_n \ \| u_n \|_{L^\infty_t L^2_x} + \| u_n \|_{L^6_{t,x}} \lesssim 1.
\end{equation}
The implicit constant here depends on $T$ and the uniform bound on $u_n(0)$ in $L^2(\R)$.  Such dependence will be tacitly permitted in all that follows.

Passing to a subsequence, we may assume that \eqref{E:P:compact} holds for some $v$, all $R>0$, and our chosen $T$.  This follows from Proposition~\ref{P:compact}.  Combining \eqref{apb} and \eqref{E:P:compact} yields
\begin{equation}\label{apb'}
\| v \|_{L^\infty_t L^2_x} + \| v \|_{L^6_{t,x}} \lesssim 1.
\end{equation}
Moreover, as $L^2(\R)$ admits a countable dense collection of $C^\infty_c$ functions, \eqref{E:P:compact} guarantees that
\begin{equation}\label{tinOmega}
u_n(t) \rightharpoonup v(t) \quad\text{weakly in $L^2(\R)$ for all $t\in\Omega$},
\end{equation}
where $\Omega\subseteq[-T,T]$ is of full measure.

We now wish to take weak limits on both sides of the Duhamel formula
\begin{equation*}
   u_n(t) = e^{it\Delta} u_n(0) - i \int_0^t e^{i(t-s)\Delta} P_{\leq N_n} F\bigl(P_{\leq N_n} u_n(s)\bigr)\,ds.
\end{equation*}
Clearly, $e^{it\Delta} u_n(0)\rightharpoonup e^{it\Delta} u_{\infty,0}$ weakly in $L^2(\R)$.  We also claim that
\begin{equation}\label{v du hamel main}
\wlim_{n\to\infty} \int_0^t e^{i(t-s)\Delta} P_{\leq N_n} F\bigl(P_{\leq N_n} u_n(s)\bigr)\,ds = \int_0^t e^{i(t-s)\Delta} F\bigl(v(s)\bigr)\,ds,
\end{equation}
where the weak limit is with respect to the $L^2(\R)$ topology.  This assertion will be justified later.  Taking this for granted for now, we deduce that
\begin{equation}\label{v du hamel}
  \wlim_{n\to\infty}  u_n(t) = e^{it\Delta} u_\infty(0) - i \int_0^t e^{i(t-s)\Delta} F(v(s))\,ds.
\end{equation}
Moreover, we observe that RHS\eqref{v du hamel} is continuous in $t$, with values in $L^2(\R)$, and that LHS\eqref{v du hamel} agrees with $v(t)$ for almost every $t$.  Correspondingly, after altering $v$ on a space-time set of measure zero, we obtain $v\in C([-T,T];L^2(\R))$ that still obeys \eqref{apb'} but now also obeys 
\begin{equation}\label{v du hamel'}
  v(t) = e^{it\Delta} u_\infty(0) - i \int_0^t e^{i(t-s)\Delta} F(v(s))\,ds \qtq{and} \wlim_{n\to\infty} u_n(t) = v(t) ,
\end{equation}
for \emph{all} $t\in[-T,T]$.  By Definition~\ref{D:solution} and Lemma~\ref{lm:loc}, we deduce that $v=u_\infty$ on $[-T,T]$ and that \eqref{E:P:weak wp} holds for $t\in [-T,T]$.

To complete the proof of Theorem~\ref{T:weak wp}, it remains only to justify \eqref{v du hamel main}.   To this end, let us fix $\psi\in L^2(\R)$.  We will divide our task into three parts.

Part 1:  By H\"older's inequality, \eqref{apb}, and the dominated convergence theorem, 
\begin{align*}
\biggl| \biggl\langle \psi,&\  \int_0^t e^{i(t-s)\Delta} \bigl[ F\bigl(P_{\leq N_n} u_n(s)\bigr) - P_{\leq N_n} F\bigl(P_{\leq N_n} u_n(s)\bigr) \bigr]\,ds \biggr\rangle \biggr| \\
&\leq  \biggl|  \int_0^t \ \Bigl\langle e^{-i(t-s)\Delta} P_{>N_n}\psi,\ F\bigl(P_{\leq N_n} u_n(s)\bigr) \Bigr\rangle \,ds\, \biggr| \\
&\leq \sqrt{T} \| P_{>N_n}\psi \|_{L^2_x} \| u_n \|_{L^6_{t,x}}^3 = o(1) \quad\text{as $n\to\infty$.}
\end{align*}

Part 2:  Let $\chi_R$ denote the indicator function of $[-R,R]$ and let $\chi_R^c$ denote the indicator of the complementary set.  Arguing much the same as for part 1, we have
\begin{align*}
\sup_n \biggl| \biggl\langle \psi,&
	\  \int_0^t e^{i(t-s)\Delta} \chi_R^c \bigl[ F\bigl(P_{\leq N_n} u_n(s)\bigr) - F\bigl(v(s)\bigr) \bigr]\,ds \biggr\rangle \biggr| \\
 &\leq T^{1/2} \| \chi_R^c e^{it \Delta} \psi \|_{L_{t,x}^6} \Bigl\{ \| v \|_{L^6_{t,x}}^2\|v\|_{L_t^\infty L_x^2} + \sup_n \| u_n \|_{L^6_{t,x}}^2\|u_n\|_{L_t^\infty L_x^2} \Bigr\} \\
 &= o(1) \quad\text{as $R\to\infty$.}
\end{align*}

Part 3: An easy application of Schur's test shows that
$$
\| \chi_R P_{\leq N_n} \chi_{2R}^c f \|_{L^p(\R)} \lesssim_\beta (N_nR)^{-\beta} \|  f \|_{L^p(\R)}
$$
for any $1\leq p\leq \infty$ and any $\beta>0$.  Correspondingly,
\begin{align*}
\bigl\| \chi_R P_{\leq N_n}  (u_n - v)  \bigr\|_{L^6_{t,x}} \lesssim \bigl\| \chi_{2R} (u_n - v) \bigr\|_{L^6_{t,x}}
	+ (N_nR)^{-\beta} \bigl\{ \| u_n\|_{L^6_{t,x}} + \| v \|_{L^6_{t,x}} \bigr\}.
\end{align*}

Using this estimate together with \eqref{apb}, \eqref{apb'}, \eqref{E:P:compact}, and the fact that $N_n\to\infty$, we deduce that
\begin{align*}
\bigl\| \chi_R \bigl[ (P_{\leq N_n}  u_n) - v \bigr] \bigr\|_{L^6_{t,x}}
	&\lesssim \bigl\| \chi_R P_{\leq N_n}  (u_n - v) \bigr] \bigr\|_{L^6_{t,x}} + \bigl\| P_{> N_n} v \bigr\|_{L^6_{t,x}} = o(1)
\end{align*}
as $n\to\infty$.  From this, \eqref{apb}, and \eqref{apb'}, we then easily deduce that 
\begin{align*}
&\limsup_{n\to\infty}\, \biggl| \biggl\langle \psi,
	\  \int_0^t e^{i(t-s)\Delta} \chi_R \bigl[ F\bigl(P_{\leq N_n} u_n(s)\bigr) - F\bigl(v(s)\bigr) \bigr]\,ds \biggr\rangle \biggr| \\
&\ \ \ \lesssim T^{1/2} \| \psi \|_{2}
	\limsup_{n\to\infty} \, \bigl\| \chi_R\bigl[ (P_{\leq N_n}  u_n) - v \bigr] \bigr\|_{L^6_{t,x}} \Bigl\{ \| v \|_{L^6_{t,x}}^2 + \| u_n \|_{L^6_{t,x}}^2 \Bigr\} = 0.
\end{align*}

Combining all three parts proves \eqref{v du hamel main} and so completes the proof of Theorem~\ref{T:weak wp}.
\end{proof}

\section{Finite-dimensional approximation}\label{S:5}

As mentioned in the introduction, to prove the non-squeezing result for the cubic NLS on the line we will prove that solutions to this equation are well approximated by solutions to a finite-dimensional Hamiltonian system.  As an intermediate step, in this section, we will prove that solutions to the frequency-localized cubic Schr\"odinger equation on the line are well approximated by solutions to the same equation on ever larger tori; see Theorem~\ref{thm:app} below.

To do this, we will need a perturbation theory for the frequency-localized cubic NLS on the torus, which in turn relies on suitable Strichartz estimates for the linear propagator.   In Lemma~\ref{lm:stri} below we exploit the observation that with a suitable inter-relation between the frequency cut-off and the torus size, one may obtain the full range of mass-critical Strichartz estimates in our setting.  We would like to note that other approaches to Strichartz estimates on the torus \cite{Bourg:1993,BourgainDemeter} are not well suited to the scenario considered here --- they give bounds on unnecessarily long time intervals, so long in fact, that the bounds diverge as the circumference of the circle diverges.

\subsection{Strichartz estimates and perturbation theory on the torus}

Arguing as in \cite[\S7]{KVZ:nsqz2d} one readily obtains frequency-localized finite-time  $L^1\to L^\infty$ dispersive estimates on the torus $\T_L:=\R/L\Z$, provided the circumference $L$ is sufficiently large.  This then yields Strichartz estimates in the usual fashion:

\begin{lem}[Torus Strichartz estimates, \cite{KVZ:nsqz2d}]\label{lm:stri}
Given $T>0$ and $1\leq N\in 2^\Z$, there exists $L_0=L_0(T,N)\geq 1$ sufficiently large so that for $L\ge L_0$,
\begin{gather*}
\|P_{\le N}^L u\|_{S([-T,T]\times\T_L)}\lsm_{q,r} \|u(0)\|_{L^2(\T_L)} + \|(i\partial_t+\Delta)u\|_{N([-T,T]\times\T_L)}.
\end{gather*}
Here, $P_{\le N}^L$ denotes the Fourier multiplier on $\T_L$ with symbol $m_{\leq 1}(\cdot / N)$.
\end{lem}

Using these Strichartz estimates one then obtains (in the usual manner) a stability theory for the following frequency-localized NLS on the torus $\T_L$:
\begin{align}\label{stabt:1}
\begin{cases}
(i\partial_t+\Delta )u =P_{\le N}^ L  F(P_{\le N}^L u),\\
u(0)=P_{\le N}^L u_0.
\end{cases}
\end{align}

\begin{lem}[Perturbation theory for \eqref{stabt:1}]\label{lm:stab} Given $T>0$ and $1\leq N\in 2^\Z$, let $L_0$ be as in Lemma~\ref{lm:stri}. Fix $L\geq L_0$ and let $\tilde u$ be an approximate solution to \eqref{stabt:1} on $[-T,T]$ in the sense that
\begin{align*}
\begin{cases}
(i\partial_t+\Delta) \tilde u=P_{\le N}^L F(P_{\le N}^L \tilde u) +e, \\
\tilde u(0)=P_{\le N}^L \tilde u_0
\end{cases}
\end{align*}
for some function $e$ and $\tilde u_0\in L^2(\T_L)$. Assume
\begin{align*}
\|\tilde u\|_{L_t^\infty L_x^2([-T,T]\times\T_L)}\le M
\end{align*}
and the smallness conditions
\begin{align*}
\|u_0-\tilde u_0\|_{L^2(\T_L)} \le \eps \qtq{and} \|e\|_{N([-T,T]\times\T_L)}\le \eps.
\end{align*}
Then if $\eps\le \eps_0(M,T)$, there exists a unique solution $u$ to \eqref{stabt:1} such that
\begin{align*}
\|u-\tilde u\|_{S([-T,T]\times\T_L)}\le C(M,T) \eps.
\end{align*}
\end{lem}

\subsection{Approximation by finite-dimensional PDE}

Fix $M>0$, $T>0$, and $\eta_n\to 0$.  Let $N_n\to \infty$ be given and let $L_n=L_n(M,T, N_n, \eta_n)$ be large constants to be chosen later; in particular, we will have $L_n\to \infty$.  Let $\btn:=\R/L_n\Z$ and let
\begin{align}\label{1241}
u_{0, n}\in \mathcal H_n:=\bigl\{f\in L^2(\btn):\,P_{>2N_n}^{L_n} f=0 \bigr\} \qtq{with} \|u_{0,n}\|_{L^2(\btn)}\le M.
\end{align}

Consider the following finite-dimensional Hamiltonian systems:
\begin{align}\label{per1}
\begin{cases}
(i\partial_t+\Delta)u_n =\pml F(\pml u_n), \qquad (t,x)\in\R\times\btn,\\
u_n(0)=u_{0,n}.
\end{cases}
\end{align}
We will show that for $n$ sufficiently large, solutions to \eqref{per1} can be well approximated by solutions to the corresponding problem posed on $\R$ on the fixed time interval $[-T,T]$.  Note that as a finite-dimensional system with a coercive Hamiltonian, \eqref{per1} automatically has global solutions.

To continue, we subdivide the interval $[\frac{L_n}4,\frac{L_n}2]$ into at least $16M^2/\eta_n$ many subintervals of length $20\frac 1{\eta_n} N_n T$.  This can be achieved so long as
\begin{align}\label{cf:1}
L_n\gg \tfrac{M^2}{\eta_n} \cdot \tfrac 1{\eta_n} N_nT.
\end{align}
By the pigeonhole principle, there exists one such subinterval, which we denote by
\begin{align*}
I_n:=[c_n-\tfrac {10}{\eta_n} N_n T, c_n+\tfrac {10}{\eta_n}N_nT],
\end{align*}
such that
\begin{align*}
\|u_{0,n}\chi_{I_n}\|_2\le \tfrac 14 \eta_n ^{1/2}.
\end{align*}
For $0\leq j\leq 4$, let $\chi_{n}^j:\R\to [0,1]$ be smooth cutoff functions adapted to $I_n$ such that
\begin{align*}
\chi_{n}^j(x)=\begin{cases}1, & x\in [c_n-L_n+\tfrac {10-2j}{\eta_n}N_nT, c_n-\tfrac {10-2j}{\eta_n}N_nT],\\
0, & x\in (-\infty, c_n-L_n+\tfrac {10-2j-1}{\eta_n}N_nT)\cup(c_n-\tfrac {10-2j-1}{\eta_n}N_nT,\infty).
\end{cases}
\end{align*}

The following properties of $\chi_n^j$ follow directly from the construction above:
\begin{align}\label{cf:1.0}
\left\{\quad\begin{gathered}
\chi_n^j\chi_n^i =\chi_n^j \quad\text{for  all $0\leq j<i\leq 4$,} \\
\|\partial_x^k \chi_n^j \|_{L^\infty} = o\bigl( (N_nT)^{-k} \bigr) \quad\text{for each $k\geq 0$}, \\
\|(1-\chi_n^j)u_{0,n}\|_{L^2(\T_n)} = o(1) \quad\text{for all $0\le j\le 4$}.
\end{gathered}\right.
\end{align}
Here and subsequently, $o(\cdot)$ refers to the limit as $n\to\infty$.  To handle the frequency truncations appearing in \eqref{per1} and \eqref{lm:1}, we need to control interactions between these cutoffs and Littlewood--Paley operators.  This is the role of the next lemma. 

\begin{lem}[Littlewood--Paley estimates]\label{L:LP est} For $L_n$ sufficiently large and all $0\leq j\leq 4$, we have the following norm bounds as operators on $L^2$\textup{:}
\begin{gather}
\|\chi_n^j (\pmn-\pml) \chi_n^j \|_{L^2(\R)\to L^2(\R)} =o(1), \label{E:P2P}\\
\|[\chi_n^j, \pml ]\|_{L^2(\btn)\to L^2(\btn)} + \|[\chi_n^j, \pmn]\|_{L^2(\R)\to L^2(\R)} =o(1), \label{E:Pcom}
\end{gather}
as $n\to\infty$.  Moreover, if $i>j$ then
\begin{gather}
\|\chi_n^j \pml (1-\chi_n^i) \|_{L^2(\btn)\to L^2(\btn)} + \|\chi_n^j \pmn (1-\chi_n^i) \|_{L^2(\R)\to L^2(\R)} =o(1). \label{E:Pmis}
\end{gather}
\end{lem}

\begin{proof}
The proof of \eqref{E:P2P} is somewhat involved; one must estimate the difference between a Fourier sum and integral and then apply Schur's test.  See \cite[\S8]{KVZ:nsqz2d} for details.

The estimate \eqref{E:Pcom} follows readily from the rapid decay of the kernels associated to Littlewood--Paley projections on both the line and the circle; specifically, by Schur's test,
\begin{gather*}
\text{LHS\eqref{E:Pcom}} \lesssim N_n^{-1} \|\nabla \chi_n^j\|_{L^\infty_x} =o(1).
\end{gather*}
We may then deduce \eqref{E:Pmis} from this.  Indeed, as $\chi_n^j\chi_n^i=\chi_n^j$ we have
\begin{gather*}
\chi_n^j \pmn (1-\chi_n^i)  = [\chi_n^j, \pmn] (1-\chi_n^i),
\end{gather*}
and analogously on the torus.
\end{proof}

Now let $\tun$ denote the solution to
\begin{align}\label{lm:1}
\begin{cases}
(i\partial_t+\Delta)\tun=\pmn F(\pmn\tun),\qquad (t,x)\in\R\times\R,\\
\tun(0,x)=\cha(x) u_{0,n}(x+L_n\Z),
\end{cases}
\end{align}
where $u_{0,n}\in L^2(\btn)$ is as in \eqref{1241}.  It follows from Lemma~\ref{lm:loc} that these solutions are global and, moreover, that they obey
\begin{align}\label{c1}
\|\partial_x^k \tun \|_{S([-T,T]\times\R)}\lsm_T M N_n^k
\end{align}
uniformly in $n$ and $k\in\{0,1\}$.  By providing control on the derivative of $\tun$, this estimate also controls the transportation of mass:

\begin{lem}[Mass localization for $\tun$]\label{lm:sm}
Let $\tun$ be the solution to \eqref{lm:1} as above. Then for every $0\leq j\leq 4$ we have
\begin{align*}
\|(1-\chi_n^j)\tun\|_{L_t^\infty L_x^2([-T,T]\times\R)} =o(1) \quad\text{as $n\to\infty$}.
\end{align*}
\end{lem}

\begin{proof}
Direct computation (cf. \cite[Lemma~8.4]{KVZ:nsqz2d}) shows
\begin{align*}
\frac{d\ }{dt} \int_{\R}|1-\chi_n^j(x)|^2 |\tun(t,x)|^2 \,dx&= - 4\Im\int_{\R^2}(1-\chi_n^j)(\nabla\chi_n^j) \overline{\tun} \nabla\tun \,dx \\
&\quad+2\Im\int_{\R^2}F(\pmn \tun)[\pmn,(1-\chi_n^j)^2]\overline{\tun} \,dx.
\end{align*}
From this, the result can then be deduced easily using \eqref{cf:1.0}, \eqref{c1}, and \eqref{E:Pcom}.
\end{proof}

With these preliminaries complete, we now turn our attention to the main goal of this section, namely, to prove the following result:

\begin{thm}[Approximation]\label{thm:app} Fix $M>0$ and $T>0$. Let $N_n\to \infty$ and let $L_n$ be sufficiently large depending on $M, T, N_n$. Assume $u_{0,n}\in \mathcal H_n$ with $\|u_{0,n}\|_{L^2(\btn)} \le M$. Let $u_n$ and $\tun$ be solutions to \eqref{per1} and \eqref{lm:1}, respectively. Then
 \begin{align}\label{pert:2.0}
 \lim_{n\to \infty} \|P_{\le 2N_n}^{L_n}(\chc \tun)-u_n\|_{S([-T,T]\times \btn)}=0.
 \end{align}
 \end{thm}
 
 \begin{rem}
Note that for $0\leq j\leq 4$ and any $t\in\R$, the function $\chi_n^j\tun(t)$ is supported inside an interval of size $L_n$; consequently, we can view it naturally as a function on the torus $\T_n$.  Conversely, the functions $\chi_n^ju_n(t)$ can be lifted to functions on $\R$ that are supported in an interval of length $L_n$.  In what follows, the transition between functions on the line and the torus will be made without further explanation.
\end{rem}

 \begin{proof}
The proof of Theorem~\ref{thm:app} is modeled on that of \cite[Theorem~8.9]{KVZ:nsqz2d}.  For brevity, we write
 \begin{align*}
 z_n:=P_{\le 2 N_n}^{L_n}(\chc \tun).
 \end{align*}
We will deduce \eqref{pert:2.0} as an application of the stability result Lemma ~\ref{lm:stab}.  Consequently, it suffices to verify the following:
 \begin{gather}
 \|z_n\|_{L_t^\infty L_x^2 ([-T,T]\times\btn)}\lsm M \mbox{ uniformly in } n,\label{per:1}\\
 \lim_{n\to \infty}\|z_n (0)-u_n(0)\|_{L^2(\btn)}=0, \label{per:2}\\
 \lim_{n\to \infty}\|(i\partial_t+\Delta)z_n-\pml F(\pml z_n)\|_{N([-T,T]\times\btn)}=0.\label{per:3}
 \end{gather}

Claim \eqref{per:1} is immediate:
 \begin{align*}
 \|z_n\|_{L_t^\infty L_x^2([-T,T]\times\btn)}\lsm \|\tun\|_{L_t^\infty L_x^2 ([-T,T]\times\R)}\lsm \|\tilde u_n(0) \|_{L_x^2(\R)}
 	\lsm \| u_{n,0} \|_{L_x^2(\btn)} \lsm M.
 \end{align*}

To prove \eqref{per:2}, we use $u_{0,n}\in \mathcal H_n$ and \eqref{cf:1.0} as follows:
 \begin{align*}
 \|z_n(0)-u_n(0)\|_{L^2(\btn)}
 &=\|\pmll (\chc u_{0,n}-u_{0, n})\|_{L^2(\btn)}\\
 &\lsm \|\chc u_{0,n}-u_{0,n}\|_{L^2(\btn)} =o(1) \qtq{as} n\to \infty.
 \end{align*}

It remains to verify \eqref{per:3}.  Direct computation gives
\begin{align*}
(i\partial_t+\Delta)z_n -\pml &F(\pml z_n)\\
&=\pmll\Bigl[2(\partial_x \chc)(\partial_x \tun) + (\Delta \chc) \tun\Bigr] \\
 &\quad+\pmll \Bigl[\chc \pmn F(\pmn \tun)-\pml F(\pml(\chc \tun))\Bigr].
\end{align*}
In view of the boundedness of $\pmll$, it suffices to show that the terms in square brackets converge to zero in $N([-T,T]\times\btn)$ as $n\to \infty$.

Using \eqref{c1} and \eqref{cf:1.0}, we obtain
\begin{align*}
\|(\partial_x \chc) (\partial_x\tun) \|_{L^1_tL^2_x([-T,T]\times\btn)}
	&\le  T \|\partial_x \chc\|_{L^\infty_x(\R)} \|\partial_x\tun \|_{L_t^\infty L^2_x([-T,T]\times\R)} = o(1),\\
\|(\Delta \chc) \tun\|_{L^1_t L^2_x ([-T,T]\times\btn)}&\le T\|\partial_x^2 \chc\|_{L^\infty_x(\R)}\|\tun\|_{L_t^\infty L_x^2([-T,T]\times\R)}= o(1), 
\end{align*}
as $n\to\infty$.

To estimate the remaining term, we decompose it as follows:
\begin{align}
\chc \pmn F(\pmn \tun)&-\pml F(\pml(\chc \tun))\notag\\
&=\chc \pmn \bigl[F(\pmn \tun)-F(\pmn(\chc\tun))\bigr]\label{per:6}\\
&\quad+\chc \pmn(1-\chd) F(\pmn(\chc \tun))\bigr]\label{per:7}\\
&\quad+\chc\pmn\chd \bigl[F(\pmn(\chc\tun))- F(\pml(\chc \tun))\bigr]\label{per:8}\\
&\quad+\chc\bigl(\pmn-\pml\bigr)\chd F(\pml(\chc\tun))\label{per:9}\\
&\quad+[\chc, \pml]\chd F(\pml (\chc\tun))\label{per:10}\\
&\quad+\pml(\chc-1) F(\pml(\chc\tun)).\label{per:11}
\end{align}

To estimate \eqref{per:6}, we use H\"older and Lemma~\ref{lm:sm}:
\begin{align*}
\|\eqref{per:6}\|_{N([-T,T]\times\btn)}
&\lsm \|F(\pmn \tun)-F(\pmn(\chc\tun))\|_{L^{6/5}_{t,x}([-T,T]\times\R)}\\
&\lsm T^{\frac 12}\|(1-\chc)\tun\|_{L^\infty_t L^2_x([-T,T]\times\R)}\|\tun\|_{L^6_{t,x}([-T,T]\times\R)}^2= o(1).
\end{align*}

We next turn to \eqref{per:7}. As
\begin{align*}
\|\eqref{per:7}\|_{L^1_tL^2_x([-T,T]\times\btn)}
&\lsm \|\chc \pmn(1-\chd)\|_{L^2(\R)\to L^2(\R)} T^{\frac 12} \|\tun\|_{L_{t,x}^6([-T,T]\times\R)}^3,
\end{align*}
it follows from \eqref{E:Pmis} and \eqref{c1} that this is $o(1)$ as $n\to\infty$.

We now consider \eqref{per:8}. Using \eqref{E:P2P} and \eqref{c1}, we estimate
\begin{align*}
\|&\eqref{per:8}\|_{L_{t,x}^{6/5}([-T,T]\times\btn)}\\
&\lsm T^{\frac 12} \|\chd(\pmn-\pml)\chc\tun \|_{L_t^\infty L_x^2([-T,T]\times\btn)}\|\chc \tun\|_{L_{t,x}^6([-T,T]\times\R)}^2\\
&\lsm T^{\frac 12} \|\chd(\pmn-\pml)\chd \|_{L^2(\R)\to L^2(\R)}\|\chc \tun\|_{L^\infty_t L^2_x ([-T,T]\times\R)}\|\tun\|_{L_{t,x}^6([-T,T]\times\R)}^2 \\
&= o(1).
\end{align*}

Next we turn to \eqref{per:9}. Using \eqref{cf:1.0},  \eqref{E:P2P}, and \eqref{c1}, we get
\begin{align*}
\|\eqref{per:9}&\|_{L^1_tL^2_x([-T,T]\times\btn)}\\
&\lsm \|\chd(\pmn-\pml)\chd\|_{L^2(\R)\to L^2(\R)}\|\che F(\pml(\chc\tun))\|_{L^1_tL^2_x([-T,T]\times\R)}\\
&\lsm o(1) \cdot T^{\frac 12} \|\tun\|_{L_{t,x}^6([-T,T]\times\R)}^3 = o(1).
\end{align*}

To estimate \eqref{per:10}, we use \eqref{E:Pcom} and \eqref{c1} as follows:
\begin{align*}
\|\eqref{per:10}&\|_{L^1_tL^2_x([-T,T]\times\btn)}\\
&\lsm\|[\chc,\pml]\|_{L^2(\btn)\to L^2(\btn)} \| \chd F(\pml(\chc\tilde u_n))\|_{L_t^1L_x^2([-T,T]\times\btn)}\\
&\lsm o(1) \cdot T^{\frac 12}  \|\tun\|_{L_{t,x}^6([-T,T]\times\R)}^3 = o(1).
\end{align*}

Finally, to estimate \eqref{per:11}, we write $\tilde u_n = \chi_n^1\tilde u_n +  (1-\chi_n^1)\tilde u_n$ and then use \eqref{c1}, \eqref{E:Pmis}, and \eqref{cf:1.0}:\begin{align*}
\|\eqref{per:11}\|_{N([-T,T]\times\btn)} &\lesssim \| (\chc-1) F(\pml(\chc\tun)) \|_{L^{6/5}_{t,x}([-T,T]\times\R)} \\
&\lsm T^{\frac12} \|(1-\chc)\pml \chi_n^1\tun\|_{L^\infty_t L^2_x([-T,T]\times\R)} \|\tilde u_n\|_{L_{t,x}^6([-T,T]\times\R)}^2 \\
&\quad+ T^{\frac12} \|(1-\chi_n^1)\tilde u_n\|_{L^\infty_tL_x^2([-T,T]\times\R)}\|\tun\|_{L_{t,x}^6([-T,T]\times\R)}^2\\
&= o(1) \qtq{as} n\to \infty.
\end{align*}

This completes the proof of the theorem.
\end{proof}

\section{Proof of Theorem \ref{thm:nsqz}}

In this section, we complete the proof of Theorem \ref{thm:nsqz}.  To this end, fix parameters $z_*\in L^2(\R)$, $l\in L^2(\R)$ with $\|l\|_2=1$, $\alpha \in \C$, $0<r<R<\infty$, and $T>0$.  Let $M:=\|z_*\|_2 + R$.  Let $N_n\to \infty$ and choose $L_n$ diverging to infinity sufficiently fast so that all the results of Section~\ref{S:5} hold.

By density, we can find $\tilde z_*, \tilde l  \in C_c^\infty(\R)$ such that
\begin{align}\label{tildeapprox}
\| z_*-\tilde z_*\|_{L^2} \le \delta \qquad\qtq{and}\qquad  \|l-\tilde l\|_{L^2} \le \delta M^{-1} \qtq{with} \|\tilde l\|_2=1,
\end{align}
for a small parameter $\delta>0$ chosen so that $\delta < (R-r)/8$.  For $n$ sufficiently large, the supports of $\tilde z_*$ and $\tilde l$ are contained inside the interval $[-L_n/2, L_n/2]$, which means that we can view $\tilde z_*$ and $\tilde l$ as functions on $\T_n=\R/L_n\Z$.  Moreover, 
\begin{equation}\label{z*n}
\|\tilde z_* - P_{\leq N_n}^{L_n} \tilde z_*\|_{L^2(\btn)}\lesssim N_n^{-1} \|\tilde z_*\|_{H^1(\btn)}  = o(1),
\end{equation}
as $n\to\infty$.  Similarly,
\begin{equation}\label{l*n}
\| P_{>2N_n}^{L_n} \tilde l\|_{L^2(\btn)}  = o(1) \quad\text{as $n\to\infty$.}
\end{equation}

Consider now the initial-value problem
\begin{equation}\label{329}
\begin{cases}
(i\partial_t+\Delta) u_n=\pml F(\pml u_n), \qquad(t,x)\in\R\times\btn,\\
u_n(0)\in \mathcal H_n=\{f\in L^2(\btn): \, P_{>2 N_n}^{L_n} f=0\}.
\end{cases}
\end{equation}
This is a finite-dimensional Hamiltonian system with respect to the standard Hilbert-space symplectic structure on $\mathcal H_n$; the Hamiltonian is
$$
H(u) = \int_{\btn} \tfrac12 |\partial_x u|^2 \pm \tfrac14 |\pml u|^4\,dx.
$$
Therefore, by Gromov's symplectic non-squeezing theorem, there exist initial data
\begin{align}\label{main:0}
u_{0,n}\in B_{\mathcal H_n}(P_{\leq N_n}^{L_n} \tilde z_*, R-4\delta)
\end{align}
such that the solution to \eqref{329} with initial data $u_n(0)=u_{0,n}$ satisfies
\begin{align}\label{main:1}
|\langle \tilde l, u_n(T)\rangle_{L^2(\btn)} -\alpha|> r + 4\delta.
\end{align}

Just as in Section~\ref{S:5} we let $\tun:\R\times\R\to\C$ denote the global solution to
\begin{align*}
\begin{cases}
(i\partial_t+\Delta) \tun=\pmn F(\pmn \tun),\\
\tun(0)=\cha u_{0,n},
\end{cases}
\end{align*}
and write $z_n:=\pmll (\chc \tun)$.  By Theorem~\ref{thm:app}, we have the following approximation result:
\begin{align}\label{main:2}
\lim_{n\to \infty}\|z_n-u_n\|_{L^\infty_t L^2_x([-T,T]\times\btn)}=0.
\end{align}

We are now ready to select initial data that witnesses the non-squeezing for the cubic NLS on the line.  By the triangle inequality, \eqref{cf:1.0}, \eqref{main:0}, \eqref{z*n}, and \eqref{tildeapprox},
\begin{align*}
\|\cha u_{0,n}-z_*\|_{L^2(\R)}
&\le \|(\cha-1) u_{0,n}\|_{L^2(\btn)} + \|u_{0,n}-P_{\leq N_n}^{L_n}\tilde z_*\|_{L^2(\btn)}\\
&\quad +\|P_{\leq N_n}^{L_n}\tilde z_*-\tilde z_*\|_{L^2(\btn)} + \| \tilde z_*-z_*\|_{L^2(\R)}\\
&\le o(1) + R-4\delta+ o(1) + \delta\le  R-\delta,
\end{align*}
provided we take $n$ sufficiently large.  Therefore, passing to a subsequence, we may assume that 
\begin{align}\label{main:3}
\cha u_{0,n}\rightharpoonup u_{0,\infty}\in B(z_*,R) \quad \text{weakly in $L^2(\R)$}.
\end{align}

Now let $u_\infty:\R\times\R\to\C$ be the global solution to \eqref{nls} with initial data $u_\infty (0)=u_{0,\infty}$. By Theorem~\ref{T:weak wp},
\begin{align*}
\tun (T) \rightharpoonup u_\infty (T) \quad \text{weakly in $L^2(\R)$}.
\end{align*}
Combining this with Lemma \ref{lm:sm}, we deduce
\begin{align*}
\chc \tun(T)\rightharpoonup u_\infty (T) \quad \text{weakly in $L^2(\R)$}.
\end{align*}
Thus, using also \eqref{l*n}, the definition of $z_n$, \eqref{main:2}, and \eqref{main:1}, we get
\begin{align*}
\bigl|\langle \tilde l, u_\infty (T)\rangle_{L^2(\R)}-\alpha\bigr |
&= \lim_{n\to \infty}\bigl|\langle \tilde l, \chc \tun(T)\rangle_{L^2(\btn)}-\alpha\bigr|\\
&= \lim_{n\to \infty}\bigl |\langle \pmll \tilde l, \chc\tun (T)\rangle_{L^2(\btn)} -\alpha\bigr|\\
&= \lim_{n\to \infty}\bigl |\langle \tilde l, z_n(T)\rangle_{L^2(\btn)} -\alpha\bigr| \\
&= \lim_{n\to \infty}\bigl |\langle \tilde l, u_n(T)\rangle_{L^2(\btn)} -\alpha\bigr| \\
&\geq r + 4\delta.
\end{align*}

Therefore, using 
$$
\|u_\infty(T)\|_{L^2(\R)} = \|u_{0,\infty}\|_{L^2(\R)} < R + \|z_*\|_{L^2(\R)} = M
$$ 
(cf. \eqref{main:3}) together with \eqref{tildeapprox}, we deduce that
\begin{align*}
\bigl |\langle l, u_\infty(T)\rangle-\alpha\bigr|\ge r + 4\delta - \|l-\tilde l\|_2\|u_\infty(T)\|_{L^2} \geq r + 3\delta > r.
\end{align*}
This shows that $u_\infty(T)$ lies outside the cylinder $C_r(\alpha, l)$, despite the fact that $u_\infty(0)\in B(z_*,R)$, and so completes the proof of the theorem.\qed

\end{document}